\documentclass[11pt]{amsart}
\usepackage{graphicx}
\usepackage{hyperref}

\newtheorem{thm}{Theorem}

\newtheorem{lem}[thm]{Lemma}

\theoremstyle{definition}

\theoremstyle{remark}
\newtheorem{rem}[thm]{Remark}

\begin{document}

\title[Infinite Toeplitz matrix]{Fixed point theorem for an infinite Toeplitz matrix}%
\author{Vyacheslav M. Abramov}%
\address{24 Sagan drive, Cranbourne North, Victoria-3977, Australia}%
\email{vabramov126@gmail.com}%

\subjclass{15B05, 40E05, 40A30, 60C99}%
\keywords{Toeplitz matrix; fixed point theorem; Tauberian theorems; generating functions; combinatorial probability}
\begin{abstract}
For an infinite Toeplitz matrix $T$ with nonnegative real entries we find the conditions, under which the equation $\boldsymbol{x}=T\boldsymbol{x}$, where $\boldsymbol{x}$ is an infinite vector-column, has a nontrivial bounded positive solution. The problem studied in this paper is associated with the asymptotic behavior of convolution type recurrence relations, and can be applied to different problems arising in the theory of stochastic processes and applied problems from other areas.
\end{abstract}
\maketitle

\section{Introduction} Let $T$ be an infinite Toeplitz matrix with nonnegative real entries
\begin{equation}\label{1}
T=\left(\begin{array}{cccc}t_0 &t_{-1} &t_{-2} &\cdots\\
t_1 &t_0 &t_{-1} &\cdots\\
t_2 &t_1 &t_0 &\cdots\\
\vdots &\vdots &\vdots &\ddots
\end{array}\right).
\end{equation}

The aim of this paper is to find the conditions under which
\begin{equation}\label{2}
\boldsymbol{x}=T\boldsymbol{x}, \quad \boldsymbol{x}=\left(\begin{array}{c}x_0\\ x_1\\ \vdots\end{array}\right)
\end{equation}
has a bounded positive solution. Since the solution $\boldsymbol{x}=\mathbf{0}$, where $\mathbf{0}$ is the infinite-dimensional vector of zeros, is a trivial solution,  we seek the only positive solutions. By positive solutions we mean the solutions $\boldsymbol{x}$ satisfying the properties $x_j\geq0$, $j=0,1\ldots$, and $\sum_{j=0}^{\infty}x_j>0$.

The general matrix equations in the form $\boldsymbol{x}=A\boldsymbol{x}$ or $\boldsymbol{x}=A\boldsymbol{x}+\boldsymbol{b}$, where $A$ is a finite or infinite positive matrix, $\boldsymbol{x}$ is the vector of unknowns and $\boldsymbol{b}$ is a known vector  have been known for a long time. They are widely used for solution of various linear equations by the fixed point method, and the area of their application is wide. They are studied by different mathematical means including functional and numerical analysis (e.g.   \cite{Kelley, Krasnoselskii et al}), while the methods that are typically used for the solution of matrix equations are iterative methods.
The detailed discussion of various iteration methods can be found in \cite{Varga}. A widely known application of the matrix equation $\boldsymbol{x}=A\boldsymbol{x}+\boldsymbol{b}$ in economy is the input-output model or Leontief model. It describes a quantitative economic model for the interdependencies between different sectors of a national economy or different regional economies \cite{Leontief}.

The areas of application of matrix equation \eqref{2}, where the matrix $T$ is specified as an infinite Toeplitz matrix differ from those mentioned above.
The possible areas of application of \eqref{2} include problems from the theory of stochastic processes that are closely related to the earlier studies in \cite{Takacs}. As well, they can include applied problems from other areas that use a similar type of analytic equations.

The problems considered in \cite{Takacs} are based on the convolution type recurrence relations. In particular, they describe the probability problems that appear as an extension of the classic ruin and ballot problems and the problems on fluctuations of sums of random variables. Further applications of the convolution type recurrence relations are known in queueing theory (e.g. \cite{Takacs1}), dams theory (e.g. \cite{Abramov}) and many other areas that also considered in \cite{Takacs}.

Equation \eqref{2} itself with finite or infinite Toeplitz matrix $T$ has not been earlier studied, and the techniques for the study of equation \eqref{2} come from the theory of the convolution type recurrence relations. Asymptotic analysis of those equations uses analytic techniques of generating functions with further application of Abelian or Tauberian theorems in their asymptotic analysis.

For the further discussion, let us recall a theorem in \cite[p. 17]{Takacs} presented here in a slightly reformulated form.

\begin{thm}\label{thm0} Let $\nu_1$, $\nu_2$,\ldots, $\nu_r$,\ldots be mutually independent, and identically distributed random variables taking nonnegative integer values, and $N_r=\sum_{j=1}^{r}\nu_j$. Let $t_j=\mathsf{P}\{\nu_1=j\}$, $j=0,1,\ldots$. If $\mathsf{E}\nu_1<1$, then
\begin{equation}\label{14}
\mathsf{P}\left\{\sup_{1\leq r< \infty}(N_r-r)<k\right\}=x_k,
\end{equation}
where $x_k=0$ for $k<0$, $x_0=1-\mathsf{E}\nu_1$, and $x_k$, $k=1,2,\ldots$, can be found by the following recurrence relation
\begin{equation}\label{5}
x_k=\sum_{j=0}^{k}t_jx_{k-j+1}.
\end{equation}
\end{thm}

The generating function of \eqref{5} has the presentation
\begin{equation}\label{6}
\sum_{k=0}^{\infty}x_kz^k=\frac{x_0\tau_0(z)}{\tau_0(z)-z}, \quad -1<z<1.
\end{equation}
where
\[
\tau_0(z)=\sum_{k=0}^{\infty}t_kz^k, \quad -1\leq z\leq 1.
\]

Notice that recurrence relation \eqref{5} can be presented in the form of the following matrix equation
\begin{equation}\label{3}
\boldsymbol{x}=T\tilde{\boldsymbol{x}},
\end{equation}
where
\begin{equation}\label{4}
T=\left(\begin{array}{cccccc}t_0  &0 &0 &0 &\cdots\\
t_1 &t_0  &0 &0 &\cdots\\
t_2 &t_1 &t_0  &0 &\cdots\\
\vdots &\vdots &\vdots &\vdots  &\ddots
\end{array}\right),
\end{equation}
and
\[
\boldsymbol{x}=\left(\begin{array}{c}x_0\\ x_1\\ \vdots\end{array}\right), \quad \tilde{\boldsymbol{x}}=\left(\begin{array}{c}x_1\\ x_2\\ \vdots\end{array}\right).
\]
However, if after an appropriate change of variables we rewrite \eqref{5} in the form:
\begin{eqnarray}
x_{-1}&=&t_{-1}x_0=1-\mathsf{E}\nu_1,\label{16}\\
x_k&=&\sum_{j=-1}^{k}t_{j}x_{k-j}, \quad k=0, 1, \ldots,\label{20}
\end{eqnarray}
then \eqref{20} can be represented as the equation $\boldsymbol{x}=T\boldsymbol{x}$, where
\begin{equation}\label{7}
T=\left(\begin{array}{cccccc}t_0 &t_{-1} &0 &0 &0 &\cdots\\
t_1 &t_0 &t_{-1} &0 &0 &\cdots\\
t_2 &t_1 &t_0 &t_{-1} &0 &\cdots\\
\vdots &\vdots &\vdots &\vdots &\vdots &\ddots
\end{array}\right), \quad \boldsymbol{x}=\left(\begin{array}{c}x_0\\ x_1\\ \vdots\end{array}\right),
\end{equation}
with boundary condition \eqref{16}.

To motivate a more general equation \eqref{2} than that specified with the particular matrix $T$ given by \eqref{7}, let us consider the following elementary extension of Theorem \ref{thm0}.
\begin{thm}\label{thm2}
Let $\nu_1$, $\nu_2$,\ldots, $\nu_r$,\ldots be mutually independent, and identically distributed random variables taking nonnegative integer values.
Denote $N_r=\sum_{j=1}^{r}\nu_j$.
If $\mathsf{E}\nu_1<n$, then
\begin{equation*}
\mathsf{P}\left\{\sup_{1\leq r<\infty}(N_{r}-nr)\leq k~|~\sup_{1\leq r<\infty}(N_{r}-nr)\geq0\right\}=x_{k}, \ k=0, 1,\ldots,
\end{equation*}
where $x_k$ satisfies the following recurrence relation
\begin{equation}\label{13}
x_{k}=\sum_{j=0}^{k+n}t_{j-n}x_{k+n-j}, 
\end{equation}
where $t_{j-n}=\mathsf{P}\{\nu_1=j\}$
\end{thm}
The proof of Theorem \ref{thm2} is given in the appendix.

\begin{rem}
The formulated theorem does not provide explicit values of $x_0$, $x_1$,\ldots, $x_{n-1}$, from which recurrence relation \eqref{13} starts. In order to have all these values, we are to present the expression for the unconditional distribution $\mathsf{P}\big\{\sup_{1\leq r<\infty}(N_{r}-nr)\leq k\big\}$. The derivation of it leads to the expressions having the complicated nature, and these unnecessary details lie out of the scope of the present paper.
\end{rem}

The expanded form of the recurrence relation \eqref{13}, which is more convenient for our purpose, is

\begin{equation}\label{15}
\begin{array}{lllllllll}
x_0 &= &t_0x_0 &+t_{-1}x_1 &+\ldots &+t_{-n}x_{n}, & & &\\
x_1 &= &t_1x_0 &+t_0x_1 &+t_{-1}x_2 &+\ldots &+t_{-n}x_{n+1}, & &\\
\vdots & &\vdots &\vdots &\vdots &\vdots &\vdots &\ddots &\\
\end{array}
\end{equation}

While for this specific problem the values $x_k$, $k=0,1,\ldots$ satisfy the conditions: $x_0\leq x_1\leq\ldots$, $\lim_{k\to\infty}x_k=1$, and the sequence $t_j$, $j=-n, -n+1,\ldots$ satisfies the property $t_{-n}+t_{-n+1}+\ldots=1$, in our further study of recurrence relations \eqref{15}, the positive values $x_0$, $x_1$,\ldots, $x_{n-1}$ can be chosen with a higher freedom, and $t_{-n}+t_{-n+1}+\ldots$ is not necessarily equal to 1.

The system of equations \eqref{15} can be represented in the form of equation \eqref{2}, where the matrix $T$ takes the form
\begin{equation}\label{8}
T=\left(\begin{array}{cccccccc}t_0 &t_{-1} &\cdots &t_{-n} &0 &0 &0 &\cdots\\
t_1 &t_0 &t_{-1} &\cdots &t_{-n} &0 &0 &\cdots\\
t_2 &t_1 &t_0 &t_{-1} &\cdots &t_{-n} &0 &\cdots\\
\vdots &\vdots &\vdots &\vdots &\vdots &\vdots &\vdots &\ddots
\end{array}\right), \quad t_{-n}>0,
\end{equation}
and the study of this equation is central in the paper.

The plan of our study is as follows. Let $n=\max\{j: t_{-j}>0\}$. In Section \ref{S2}, we formulate the main results. In Section \ref{S3}, we derive the explicit representations for the generating function for equation \eqref{2} in the case $n=1$ and then in the case of arbitrary fixed $n$. As well, we discuss the existence of a positive solution of equation \eqref{2} under the assumptions on the entries of the matrix $T$.
In Section \ref{S4}, we prove the main results of this paper.

\section{Main results}\label{S2}

The theorem below assumes that
$n=\max\{j: t_{-j}>0\}<\infty.$
Under this assumption there are infinitely many positive solutions of equation \eqref{2}. However, if the first $n$ positive elements $x_0$, $x_1$,\ldots, $x_{n-1}$ of the vector $\boldsymbol{x}$ are fixed, then the recurrence relations provide a unique solution of equation \eqref{2}.
Denote
\[
\tau_{-n}(z)=\sum_{k=0}^{\infty}t_{k-n}z^k, \quad -1\leq z\leq1.
\]

\begin{thm}\label{thm1} Assume that $n=\max\{j: t_{-j}>0\}<\infty$,
and
\begin{equation}\label{24}
\frac{\mathrm{d}}{\mathrm{d}z}\sqrt[n]{\tau_{-n}(z)} \quad \text{increases.}
\end{equation}

\begin{enumerate}
\item [(i)] If $\sum_{k=0}^{\infty}t_{k-n}>1$, then all positive solutions are bounded, and \[\lim_{k\to\infty}x_k=0.\]
\item [(ii)] If $\sum_{k=0}^{\infty}t_{k-n}=1$, then all positive solutions are bounded if and only if
\begin{equation}\label{19}
\sum_{k=1}^{\infty}kt_{k-n}<n.
\end{equation}
In the case $n=1$, if $\sum_{j=1}^{\infty}jt_{j-1}<1$, then  it satisfies the property
\begin{equation}\label{21}
\lim_{k\to\infty}x_k=\frac{x_0t_{-1}}{1-\sum_{j=1}^{\infty}jt_{j-1}}.
\end{equation}
\item [(iii)] If $\sum_{k=0}^{\infty}t_{k-n}<1$, then any positive solution is unbounded.
\end{enumerate}
\end{thm}

\begin{rem}
Here we provide an elementary example for which \eqref{24} is satisfied. Let $t_{k-n}=\mathrm{e}^{-a}a^k/k!$, where $a$ is a positive constant. Then $\tau_{-n}(z)=\mathrm{e}^{a(z-1)}$, $\sqrt[n]{\tau_{-n}(z)}=\mathrm{e}^{\frac{a}{n}(z-1)}$, and
\[
\frac{\mathrm{d}}{\mathrm{d}z}\sqrt[n]{\tau_{-n}(z)}=\frac{a}{n}\mathrm{e}^{\frac{a}{n}(z-1)} \quad \text{increases}.
\]
Apparently that if in the aforementioned example we assume that $a$ is a non-decreasing function of $z$, then \eqref{24} will be satisfied. If $a$ is not a non-decreasing function of $z$, then \eqref{24} can be violated.
\end{rem}

\begin{rem}In the case when $n=\max\{j: t_{-j}>0\}$ does not exist, we set $n=\infty$. Then the corresponding results can be reformulated as asymptotic theorems for $n\to\infty$, and in computation of $T\boldsymbol{x}$ we are to use the required convergence theorems. Then condition \eqref{19} and \eqref{24}
will be transformed as follows.

Instead of condition \eqref{19} we shall require that
\[
\lim_{n\to\infty}\frac{1}{n}\sum_{k=1}^{\infty}kt_{k-n}<1.
\]
Instead of \eqref{24}, we shall assume that for all large $n$
\[
\frac{\mathrm{d}}{\mathrm{d}z}\sqrt[n]{\tau_{-n}(z)}
\]
are increasing functions. Note that a slightly stronger assumption than \eqref{24} is the requirement that
\[
\frac{\mathrm{d}}{\mathrm{d}z}\log{\tau_{-n}(z)}
\]
increases.
\end{rem}

\section{Case studies of equation \eqref{2}}\label{S3}

\subsection{The case when $T$ is presented by \eqref{7}}\label{S3.1}

In the particular case when $T$ is presented by \eqref{7} we have the recurrence relations are defined by \eqref{20}. To make further derivations clearer, we present them in the expanded form:
\begin{equation}\label{9}
\begin{array}{lllllll}
x_0 &= &t_0x_0 &+t_{-1}x_1,& & & \\
x_1 &= &t_1x_0 &+t_0x_1 &+t_{-1}x_2,& &\\
\vdots & &\vdots &\vdots &\vdots &\ddots &
\end{array}
\end{equation}
Using generating functions and combining the terms of \eqref{17} by columns, we obtain
\begin{equation}\label{10}
\begin{aligned}
\chi_{-1}(z)=\sum_{k=0}^\infty x_kz^k&=z^0(t_0x_0+t_1x_0z+t_2x_0z^2+\ldots)\\
+&z^0(t_{-1}x_1+t_0x_1z+t_1x_1z^2+\ldots)\\
+&z^1(t_{-1}x_2+t_0x_2z+t_1x_2z^2+\ldots)\\
+&z^2(t_{-1}x_3+t_0x_3z+t_1x_3z^2+\ldots)\\
+&\ldots\\
&=\frac{1}{z}\big(\tau_{-1}(z)\chi_{-1}(z)-t_{-1}x_0\big), \quad -1\leq z\leq 1.
\end{aligned}
\end{equation}
From \eqref{10} we obtain
\begin{equation}\label{17}
\chi_{-1}(z)=\frac{t_{-1}x_0}{\tau_{-1}(z)-z}, \quad -1<z<1.
\end{equation}

As in \eqref{6}, the generating function $\chi_{-1}(z)$ depends on the choice of $x_0$. The subindex $(-1)$ in the functions $\chi_{-1}(z)$ and $\tau_{-1}(z)$ is $\max\{j: t_{-j}>0\}$.

\subsection{The case when $T$ is presented by \eqref{8}}\label{S.2}

In the case when the system of equations is presented by \eqref{8}, the system of recurrence relations is as follows:
\begin{equation}\label{11}
\begin{array}{lllllllll}
x_0 &= &t_0x_0 &+t_{-1}x_1 &+\ldots &+t_{-n}x_n, & & &\\
x_1 &= &t_1x_0 &+t_0x_1 &+t_{-1}x_2 &+\ldots &+t_{-n}x_{n+1}, & &\\
\vdots & &\vdots &\vdots &\vdots &\vdots &\vdots &\ddots &\\
\end{array}
\end{equation}

Similarly to that given before, we are to derive the expression for the generating function $\chi_{-n}(z)=\sum_{k=0}^{\infty}x_kz^k$, where subindex $(-n)$ of the function $\chi_{-n}(z)$ is $\max\{j: t_{-j}>0\}$. The derivation of $\chi_{-n}(z)$ is provided by the same scheme as that in \eqref{10}. The expression for $\chi_{-n}(z)$ is
\begin{equation}\label{18}
\chi_{-n}(z)=\frac{\sum_{k=0}^{n-1}x_k\sum_{j=k+1}^{n}t_{-j}z^{n-j+k}}{\tau_{-n}(z)-z^n}, \quad -1<z<1.
\end{equation}

\subsection{Existence of a positive solution of equation \eqref{2} under different assumptions on the matrix $T$}\label{S3.3}
The generating function $\chi_{-n}(z)$ in \eqref{18} is expressed via arbitrary chosen $n$ positive parameters $x_0$, $x_1$,\ldots, $x_{n-1}$ such that a solution of \eqref{2} is positive.

We now consider the two cases $\sum_{k=0}^{\infty}t_{k-n}\leq1$ and $\sum_{k=0}^{\infty}t_{k-n}>1$.

We demonstrate first that in the case $\sum_{k=0}^{\infty}t_{k-n}\leq1$, positive parameters $x_0$, $x_1$,\ldots, $x_{n-1}$ guaranteeing a positive solution of equation \eqref{2} always exist.

Indeed, setting $x_0=x_1=\ldots=x_{n-1}$, we obtain
\begin{equation}\label{26}
x_n=\frac{(1-t_0-t_{-1}-\ldots-t_{-n+1})x_{n-1}}{t_{-n}}\geq x_{n-1}.
\end{equation}
Then,
\[
\begin{aligned}
x_{n+1}&=\frac{(1-t_1-t_0-\ldots-t_{-n+2})x_{n-1}-t_{-n+1}x_n}{t_{-n}}\\
&\geq\frac{(1-t_1-t_0-\ldots-t_{-n+1})x_{n-1}}{t_{-n}}\\
&= x_n.
\end{aligned}
\]
This procedure continues, and by induction we obtain $x_{n-1}\leq x_{n}\leq\ldots$. So, the sequence $x_k$, $k=0,1,\ldots$ is monotone increasing.
Note that in the case $n=1$, due to the same derivation as above, if $\sum_{k=0}^{\infty}t_{k-1}\leq1$, then  the sequence $x_k$, $k=0,1,\ldots$ is always monotone increasing.

Within the same case $\sum_{k=0}^{\infty}t_{k-n}\leq1$, assume now that $x_0$, $x_1$,\ldots, $x_{n-1}$ are chosen in some free way under which the solution of \eqref{2} is positive.
Let $x_{m_0}$, $0\leq m_0\leq n-1$, be a largest among $x_0$, $x_1$,\ldots, $x_{n-1}$. Then among the values $x_n$, $x_{n+1}$,\ldots, $x_{n+m_0}$ there is a value that is not smaller than $x_{m_0}$. Indeed, if we assume that $x_n< x_{m_0}$, $x_{n+1}< x_{m_0}$,\ldots, $x_{n+m_0-1}< x_{m_0}$, then for $x_{n+m_0}$ we obtain
\[
\begin{aligned}
x_{n+m_0}&=\frac{x_{m_0}(1-t_0)-\sum_{\substack{0\leq j\leq n+m_0-1\\ j\neq m_0}}x_jt_{m_0-j}}{t_{-n}}\\
&>\frac{x_{m_0}\left(1-\sum_{j=0}^{n+m_0-1}t_{m_0-j}\right)}{t_{-n}}\\
&> x_{m_0}.
\end{aligned}
\]
Similarly, if $m_1=\{\min n: x_n\geq x_{m_0}\}$, then among $x_{m_1}$, $x_{m_1+1}$,\ldots, $x_{m_1+n}$ there exists the value that is not smaller than $x_{m_1}$ that is denoted by $x_{m_2}$. This procedure can be continued, and one finds the limit $x^*=\lim_{i\to\infty}x_{m_i}$, that is $\limsup_{k\to\infty}x_k$. The conditions under which this upper limit is finite follows from Theorem \ref{thm1}, which is proved in the next section. Note that $m_{i+1}-m_i\leq n$ for any $i\geq1$.

 In the case when $\sum_{k=0}^{\infty}t_{k-n}>1$ a positive solution of equation \eqref{2} generally does not exist. For instance, if $t_0>1$, then from the first equation of \eqref{11} we obtain
\[
t_{-1}x_1+\ldots+t_{-n}x_n<0,
\]
that means that at least one of $x_1$, $x_2$,\ldots, $x_n$ must be negative.

Below we demonstrate a particular case of the matrix $T$ under which a positive solution of \eqref{2} does exist.

Indeed, assume that $\sum_{k=1}^{\infty}t_{k-n}<1$, and as earlier set $x_0=x_1=\ldots=x_{n-1}$, $x_0>0$. If $\sum_{k=0}^{\infty}t_{k-n}>1$, we set $t_{-n}^\prime=1-\sum_{k=1}^{\infty}t_{k-n}$, and denote $t_{-n}=ct_{-n}^\prime$, $c>1$.

Now consider two equations. The first original equation $\boldsymbol{x}=T\boldsymbol{x}$, in which $\sum_{k=0}^{\infty}t_{k-n}>1$, and the second one $\boldsymbol{u}=T^\prime\boldsymbol{u}$ in which $\sum_{k=1}^{\infty}t_{k-n}+t_{k-n}^\prime=1$, and $\boldsymbol{u}=\left(\begin{matrix}u_0\\ u_1\\ \vdots\end{matrix}\right).$ That is, the elements $t_j$, $j=-n+1, j=-n+2,\ldots$ in both matrices are the same, but the element $t_{-n}$ in the matrix $T$ and the corresponding element $t_{-n}^\prime$ in the matrix $T^\prime$ are distinct. For positive initial values set $u_0=u_1=\ldots=u_{n-1}=x_0=x_1=\ldots=x_{n-1}$. Then, a positive solution of the equation $\boldsymbol{u}=T^\prime\boldsymbol{u}$ exists and not decreasing, i.e. $u_{n-1}\leq u_n\leq\ldots$. From the equations
\begin{eqnarray*}
u_k&=&t_ku_0+t_{k-1}u_1+\ldots+t_0u_k+\ldots+t_{-n+1}u_{k+n-1}+t_{-n}^\prime u_{k+n}\\
x_k&=&t_kx_0+t_{k-1}x_1+\ldots+t_0x_k+\ldots+t_{-n+1}x_{k+n-1}+t_{-n}^\prime (x_{k+n}/c),
\end{eqnarray*}
we obtain a clear dependence between the terms $x_k$, $k=n, n+1\ldots$ and corresponding terms $u_k$, $k=n, n+1\ldots$ through the constant $c$. Since $c>1$, then $x_k<u_k$ and $x_k$ is positive.
Hence a solution of equation \eqref{2} is positive.

With similar arguments, one can prove the existence of a positive solutions under the more general assumption: $\sum_{k=0}^{\infty}t_{k}<1$ and $\sum_{k=0}^{\infty}t_{k-n}>1$. Under these assumption we set $\sum_{k=1}^{n}t_k^\prime=1-\sum_{k=0}^{\infty}t_k$ and $t_{-k}=ct_{-k}^\prime$, where $c>1$ is a unique constant that is found from this system of equations. The following arguments are similar to those given above for the particular case.

\section{Proof of Theorem \ref{thm1}}\label{S4}
\subsection{Lemmas}
The proof of the major statements of the theorem are based on the Tauberian theorem of Hardy and Littlewood \cite{Hardy, HL}, the formulation of which is as follows.

\begin{lem}\label{lem1} Let the series
\[
\sum_{j=0}^{\infty}a_jz^j
\]
converges for $-1<z<1$ and suppose there exists $\gamma>0$ such that
\[
\lim_{z\uparrow1}(1-z)^\gamma\sum_{n=0}^{\infty}a_jz^j=A.
\]
Suppose also that $a_j\geq0$. Then, as $N\to\infty$,
\[
\sum_{j=0}^Na_j\asymp\frac{A}{\Gamma(1+\gamma)}N^\gamma,
\]
where $\Gamma(x)$ is Euler's Gamma-function.
\end{lem}

In addition to this lemma we need one more lemma presented below, where we assume that $\sum_{k=0}^{\infty}t_{k-n}=1$.

\begin{lem}\label{lem3} Let $w$ be a positive value. Consider the equation
\begin{equation}\label{23}
z^n=w\tau_{-n}(z),
\end{equation}
and assume that \eqref{24} is fulfilled.

\begin{enumerate}
\item [$(a_1)$] If $w=1$ and $\gamma=\sum_{k=1}^{\infty}kt_{k-n}\leq n$, then there are no roots of equation \eqref{23} in the interval $(0,1)$.
\item [$(a_2)$] If $w=1$ and $\gamma=\sum_{k=1}^{\infty}kt_{k-n}> n$, then there is a root of equation \eqref{23} in the interval $(0,1)$.
\item [$(a_3)$] If $w>1$, then there are no roots of equation \eqref{23} in the interval $(0,1)$.
\item [$(a_4)$] If $w<1$, then there is a root of equation \eqref{23} in the interval $(0,1)$.
\end{enumerate}
\end{lem}

\begin{proof} From \eqref{23} we have the equation $z=\sqrt[n]{w\tau_{-n}(z)}$. The function $\sqrt[n]{w\tau_{-n}(z)}$ is increasing since $\tau_{-n}(z)$ is increasing, and its derivative, according to assumption of the lemma, is increasing as well. Taking into account that $t_{-n}>0$, in cases $(a_1)$ and $(a_2)$ we easily arrive at the required statements, since the difference $z-\sqrt[n]{\tau_{-n}(z)}$ in point $z=0$ is negative and in point $z=1$ is zero. The derivative of this difference in point $z=1$ is equal to $1-(1/n)\tau_{-n}^\prime(1)$. It is nonnegative in case $(a_1)$ and strictly negative in case $(a_2)$. In case $(a_3)$ the required result follows from the fact that under condition \eqref{24} the difference $z-w\sqrt[n]{\tau_{-n}(z)}$ is negative for all $0\leq z\leq 1$. In case $(a_4)$ the result trivially follows, since the differences $z-w\sqrt[n]{\tau_{-n}(z)}$ in points $z=0$ and $z=1$ are of opposite signs.
\end{proof}

\subsection{Proof of the theorem}

Under assumption (i) of the theorem, we have
\begin{equation}\label{25}
\sum_{k=0}^{\infty}t_{k-n}=w\sum_{k=0}^{\infty}t_{k-n}^\prime,
\end{equation}
where $w>1$ and $\sum_{k=0}^{\infty}t_{k-n}^\prime=1$. Then, according to statement $(a_3)$ of Lemma \ref{lem3}, the denominator of the fraction on the right-hand side of \eqref{18} is nonzero for all $z\in[0,1]$ and hence the series $\chi_{-n}(z)$ is continuous in  $[0,1]$. As $z\to1$, we have $\sum_{k=0}^{\infty}x_k=\lim_{z \uparrow1}\chi_{-n}(z)<\infty$. Then
$
\lim_{k\to\infty}x_k=0,
$
and the statement of the theorem under assumption (i) is proved.

Under assumption (ii), Lemma \ref{lem1} is applied with $\gamma=1$. We have
\begin{equation}\label{22}
\lim_{z\uparrow1}(1-z)\chi_{-n}(z)=\lim_{z\uparrow1}(1-z)\frac{\sum_{k=0}^{n-1}x_k\sum_{j=k+1}^{n}t_{-j}z^{n-j+k}}{\tau_{-n}(z)-z^n}.
\end{equation}

If condition \eqref{19} of the theorem is satisfied, then the L'Hospital rule yields
\[
\lim_{z\uparrow1}(1-z)\chi_{-n}(z)=\frac{\sum_{k=0}^{n-1}x_k\sum_{j=k+1}^{n}t_{-j}}{n-\sum_{k=1}^{\infty}kt_{k-n}}.
\]
Then the conditions of Lemma \ref{lem1} are satisfied, and according to that lemma for large $N$ we have
\begin{equation}\label{27}
\sum_{j=0}^Nx_j\asymp\frac{\sum_{k=0}^{n-1}x_k\sum_{j=k+1}^{n}t_{-j}}{n-\sum_{k=1}^{\infty}kt_{k-n}}N.
\end{equation}
Next, it was shown in Section \ref{S3.3} that there is an increasing sequence of indices $m_0$, $m_1$, \ldots such that $\lim_{i\to\infty}x_{m_i}=\limsup_{k\to\infty}x_k$, and for any $i$, $m_{i+1}-m_i\leq n$. This enables us to conclude that $\sum_{j=0}^Nx_{m_j}=O(N)$, and hence
\eqref{27} implies
$
\limsup_{k\to\infty}x_k<\infty.
$

In the particular case $n=1$, the sequence $x_k$, $k=0,1,2,\ldots$ is non-decreasing (see Section \ref{S3.3}), and hence there is the limit of this sequence as $k\to\infty$. This limit is finite, if $\sum_{k=1}^{\infty}kt_{k-1}<1$, and according to Abel's theorem
\[
\lim_{k\to\infty}x_k=\lim_{z\uparrow1}(1-z)\chi_{-1}(z)=\frac{x_0t_{-1}}{1-\sum_{j=1}^{\infty}jt_{j-n}}.
\]
Relation \eqref{21} follows.

If $\sum_{k=1}^{\infty}kt_{k-n}=n$, then the L'Hospital rule yields infinite value in limit. So, under the assumption $\sum_{k=1}^{\infty}kt_{k-n}=n$, the sequence $x_0$, $x_1$,\ldots diverges for any positive initial values of $x_0$, $x_1$,\ldots, $x_{n-1}$. If $\sum_{k=1}^{\infty}kt_{k-n}>n$, then according to statement $(a_2)$ of Lemma \ref{lem3}, the fraction of the right-hand side of \eqref{22} has a pole, and hence the sequence $x_0$, $x_1$,\ldots diverges. Statements (ii) of the theorem are proved.

Under assumption (iii) of the theorem we have \eqref{25}, where $w<1$ and $\sum_{k=0}^{\infty}t_{k-n}^\prime=1$.
Then, according to statement $(a_4)$ of Lemma \ref{lem3}, the fraction of the right-hand side of \eqref{22} has a pole. Hence the sequence $x_0$, $x_1$,\ldots diverges in this case as well. Statement (iii) follows.
The theorem is proved.

\appendix
\section{Proof of Theorem \ref{thm2}}

Let
\begin{equation}\label{12}
\mathsf{P}\left\{\sup_{1\leq r<\infty}(N_{r}-nr)\leq k~|~\sup_{1\leq r<\infty}(N_{r}-nr)\geq0\right\}=x_k.
\end{equation}
Note first that the condition $\mathsf{E}\nu_1<n$ guarantees the existence of the distribution of $\sup_{1\leq r<\infty}(N_{r}-nr)$, and hence the existence of the conditional distribution in \eqref{12}.

Following the arguments similar to those in \cite[p. 17, formula (22)]{Takacs}, by the formula for the total probability for $k\geq0$ we obtain
\[
\begin{aligned}
\mathsf{P}\{N_r&\leq r+k \quad \text{for} \quad r=1,2,\ldots,m+1\}\\
&=\sum_{j=0}^{k+n}t_{j-n}\mathsf{P}\{N_r\leq r+k+n-j \quad \text{for} \quad r=1,2,\ldots,m\},
\end{aligned}
\]
where $t_{j-n}=\mathsf{P}\{\nu_1=j\}$, $j=0,1,\ldots$.

Now, letting $m$ to increase to infinity, we arrive at the required recurrence relation.


\begin{thebibliography}{10}
\bibitem{Abramov} Abramov, V. M. (2019). Optimal control of a large dam with compound Poisson input and costs depending on water levels. \textit{Stochastics}. 91(3): 433--483. \url{doi.org/10.1080/17442508.2018.1551395}
\bibitem{Hardy} Hardy, G. H. (2000). \emph{Divergent Series}, 2nd ed. Providence: AMS Chelsea Publishing.
\bibitem{HL} Hardy, G. H., Littlewood, J. E. (1914). Tauberian theorems concerning power series and Direchlet's series whose coefficients are positive. \emph{Proc. London Math. Soc.} 13: 174--191. \url{doi.org/10.1112/plms/s2-13.1.174}
\bibitem{Kelley} Kelley, C. T. (1995). \emph{Iterative Methods for Linear and Nonlinear Equations}. Philadelphia: SIAM.
\bibitem{Krasnoselskii et al} Krasnosel'skii, M. A., Vainikko, G. M. Zabreiko, P. P., Rutitskii, Ya. B., Stetsenko, V. Ya. (1972). \textit{Approximate Solutions of Operator Equations}. Groningen: Wolters-Noordhoff Publishing Co.
\bibitem{Leontief} Leontief, Wassily. (1986). \emph{Input-Output Economics}, 2nd ed., Oxford: Oxford University Press.

\bibitem{Takacs} Tak\'acs, L. (1967). \emph{Combinatorial Methods in the Theory of Stochastic Processes}. New York: John Wiley \& Sons.
\bibitem{Takacs1} Tak\'acs, L. (1976). On the busy periods of single-server queues with Poisson input and general service times. \emph{Operat. Res.} 24(3): 564--571. \url{doi.org/10.1287/opre.24.3.564}
\bibitem{Varga} Varga, R. S. (1965). Iterative methods for solving matrix equations. \emph{Amer. Mathem. Monthly}. 72(2): 67--74. \url{doi.org/10.1080/00029890.1965.11970700}
\end{thebibliography}
\end{document}